\documentclass[11pt, letterpaper]{amsart}
\usepackage{graphicx, amssymb, color}
\usepackage{hyperref}

\newtheorem{thm}{Theorem}[section]

\newtheorem{lem}[thm]{Lemma}

\theoremstyle{definition}

\theoremstyle{remark}
\newtheorem{rem}[thm]{Remark}

\numberwithin{equation}{section}

\newcommand{\abs}[1]{\left\vert#1\right\vert}
\newcommand{\set}[1]{\left\{#1\right\}}
\newcommand{\Real}{\mathbb R}
\newcommand{\Natural}{\mathbb N}

\newcommand{\such}{\ | \ }
\newcommand{\tir}{t \in \Real_+}

\newcommand{\prob}{\mathbb{P}}
\newcommand{\Exp}{\mathcal E}

\newcommand{\expec}{\mathbb{E}}

\newcommand{\F}{\mathcal{F}}

\newcommand{\cadlag}{c\`adl\`ag}

\newcommand{\ud}{\mathrm d}
\newcommand{\inner}[2]{\left \langle #1 , #2 \right \rangle}

\newcommand{\Pre}{\mathcal{P}}
\newcommand{\Y}{\mathcal{Y}}

\newcommand{\pare}[1]{\left(#1\right)}
\newcommand{\bra}[1]{\left[#1\right]}
\newcommand{\dbra}[1]{[\kern-0.15em[ #1 ]\kern-0.15em]}
\newcommand{\dbraco}[1]{[\kern-0.15em[ #1 [\kern-0.15em[}
\newcommand{\dbraoc}[1]{]\kern-0.15em] #1 ]\kern-0.15em]}

\newcommand{\Q}{\mathcal{Q}}

\newcommand{\bF}{\mathbf{F}}

\newcommand{\indic}{\mathbf{1}}

\newcommand{\nin}{n \in \Natural}
\newcommand{\kin}{k \in \Natural}

\newcommand{\iii}{i \in \set{1, \ldots, d} }
\newcommand{\jii}{j \in \set{1, \ldots, d} }
\newcommand{\dfn}{ := }

\newcommand{\hY}{\widehat{Y}}

\newcommand{\hV}{\widehat{V}}

\newcommand{\Lc}{\mathcal{L}^\mathsf{d}}

\newcommand{\cA}{\check{A}}

\begin{document}

\title[optional decomposition for continuous semimartingales under arbitrary filtrations]{Optional decomposition for continuous semimartingales under arbitrary filtrations}%

\author{Ioannis Karatzas}
\thanks{Research   supported in part by  the National Science Foundation under  grant   NSF-DMS-14-05210, and the Marie Curie Career Integration Grant (FP7-PEOPLE-2012-CIG), project number 334540. We are greatly indebted to Tomoyuki Ichiba and Johannes Ruf, for their careful reading of the manuscript and their many suggestions. 
	}  %
\address{Ioannis Karatzas: Department of Mathematics, Columbia University, New York, NY 10027, USA;   Intech Investment Management, One Palmer Square, Suite 441, Princeton, NJ 08542, USA}%
\email{ik@math.columbia.edu, ik@enhanced.com}%

\author{Constantinos Kardaras}%
\address{Constantinos Kardaras: Department of Statistics, London School of Economics,   10 Houghton Street, London, WC2A 2AE, England}%
\email{k.kardaras@lse.ac.uk}%


\date{\today}%
\begin{abstract}
We present an elementary treatment of the Optional Decomposition Theorem  for continuous semimartingales and general filtrations. This treatment does not assume the existence of   equivalent local martingale measure(s), only that of   strictly positive local martingale  deflator(s). 
\end{abstract}

\maketitle


\section*{Introduction} \label{sec:intro}

The \emph{Optional Decomposition Theorem} (ODT) is an important result in the field of Stochastic Analysis, and more particularly in Mathematical Finance. In one of its most ``classical'' forms, following \cite{MR1402653}, the ODT can be stated as follows. For some $d \in \Natural$, let $X$ be a $\Real^d$-valued locally bounded semimartingale on a given filtered probability space $(\Omega, \mathcal{F}, \mathbb{P})$,    $\bF = \{ \mathcal{F} (t) \}_{\tir}$, and assume that $\Q\,$, the collection  of probability measures that are equivalent to $\,\mathbb{P}\,$ and bestow the local martingale property to $X$,  is non-empty. Then, a given nonnegative process $V$ is a supermartingale under \emph{all} probabilities in $\Q$, if and only if it admits the ``optional'' decomposition 
\begin{equation}  \label{eq:OD} \tag{OD}
V = V(0) + \int_0^\cdot \inner{H(t)}{ \ud X(t)} - C\,;
\end{equation}
here $H$ is a predictable $X$-integrable process, and $C$ is a nondecreasing right-continuous adapted process with $C(0) = 0$.

The representation \eqref{eq:OD} is relevant in the setting of Mathematical Finance. Indeed, suppose the components of $X$ represent (discounted) prices of assets in a financial market. If $H = (H_i)_{\iii}$ is the investment strategy of an agent in the market, where $H_i$ stands for the units of asset $i$ held in the portfolio for all $\iii$, and $C$ measures the agent's aggregate consumption, then $V$ in \eqref{eq:OD} corresponds to the generated wealth-consumption process starting with initial capital $V(0)$. The ODT offers an illuminating ``dual'' characterization of all such wealth-consumption processes, as supermartingales under all \emph{equivalent local martingale measures}   
of $X$. Using this characterization, the ODT establishes the \emph{superhedging duality} via use of dynamic programming techniques in   non-Markovian settings.

Stochastic controllability results, similar to the ODT and obtained via martingale methods, can be traced as far back as \cite{MR0319642} and, in the context of Mathematical Finance, to \cite{MR1089152}. A version of the ODT when $X$ is driven by Brownian motion under quasi-left-continuous filtrations appears in \cite{MR1311659}. 
The first paper to treat the ODT for general locally bounded semimartingales  is \cite{MR1402653}, where functional (convex) analytic methods and results from \cite{MR1304434} were employed. In  \cite{MR1469917}, the more general case of constraints on investment is considered, using essentially similar arguments. In \cite{MR1804665}, the assumption of local boundedness on the semimartingale integrator $X$ is dropped; more importantly, the authors avoid infinite-dimensional convex analysis, by following an alternative approach via predictable characteristics; this involves Lagrange multipliers,     separating hyperplane arguments in Euclidean space, and measurable selections. Although the treatments of the ODT in the aforementioned papers are quite general, they do require a significant level of sophistication
; indeed, they  involve either use of difficult functional-analytic results, or deep knowledge of the General Theory of   Processes as presented,  e.g., in \cite[Chapters I and II]{MR1943877}.

The present paper offers a rather elementary proof of the ODT for continuous-path semimartingale integrator $X$ but \emph{arbitrary} filtrations. Instead of assuming that $\Q \neq \emptyset$, it uses the more ``localized'' assumption that the class  $\Y$ of   \emph{strictly positive local martingale deflators} is non-empty. This assumption $\Y \neq \emptyset$ is both more general and more descriptive: it  allows for an equivalent structural characterization of its validity by inspecting the local drift and local covariation processes of $X$, as mentioned in Theorem \ref{thm:struct}. (In fact, \cite{MR1651229} treats the ODT using the condition $\Y \neq \emptyset$.) The important pedagogical element of the paper is that it avoids use of functional analysis and predictable characteristics in order to obtain the ODT. Since arbitrary filtrations support local martingales with potential jumps at both accessible and totally inaccessible times, it is impossible to  avoid entirely the use of certain results from the general theory of Stochastic Processes. However, we feel that the path taken here is as elementary as possible. Although some intersection with previous work exists (notably, \cite{MR1311659}, as well as \cite{MR2995513} which deals with continuous asset prices and continuous filtrations), we believe that the present treatment is more straightforward.

\section{The Setting} \label{sec:set}

\subsection{Preliminaries}
We shall work on a   probability space $(\Omega, \mathcal{F}, \mathbb{P})$, endowed with a filtration  $\bF = \{ \mathcal{F} (t) \}_{\tir}$  which satisfies the usual hypotheses of right-continuity and augmentation by null sets. We stress that no further assumption is made on the filtration.

Let $X = (X_i)_{\iii}$ be a $d$-dimensional semimartingale with continuous paths. We write $X = A + M$ for the Doob-Meyer decomposition of $X\,$; here $A$ is a $d$-dimensional    process with continuous paths of finite variation and $A(0) = 0$, and $M$ is a  $d$-dimensional local martingale with continuous paths. 

For $\iii$, we shall denote by $\cA_i$ the  process of finite first variation associated with $A_i$. Upon defining $G \dfn \sum_{i =1}^d \pare{\cA_i + [M_i, M_i]}$, it follows that there exist a $d$-dimensional predictable process $a$, and a predictable process $c$ taking values in the set of nonnegative-definite matrices, such that $\,A = \int_0^\cdot a(t) \ud G(t)\,$ and $\,[M_i, M_j] = \int_0^\cdot c_{ij}(t) \ud G(t)\,$ hold for $\iii$ and $\jii$.

\smallskip

We shall denote by $\,\Pre\,$   the predictable $\sigma$-field on $\Omega \times \Real_+$, and by $\,\prob \otimes G\,$ the measure on the product measurable space $\pare{\Omega \times \Real_+, \Pre}$ which satisfies $(\prob \otimes G) [J] = \expec \bra{ \int_0^\infty \indic_J (t) \ud G(t)}$ for all $J \in \Pre$.

Let $\Pre(X)$ denote the collection of all $d$-dimensional, predictable and $X$-integrable processes. A given     $d$-dimensional predictable  
process $H$ belongs to $\Pre(X)$ if and only if both processes $\int_0^\cdot \abs{ \inner{H(t)}{a(t)} } \ud G(t)$ and $\int_0^\cdot \inner{H(t)}{c(t)H(t)} \ud G(t)$ are finitely-valued.

 \smallskip
We shall be using the   notation  
\[
\Exp(Z)\,:=\, \exp \Big( Z - \frac{1}{\,2\,} \big[ Z, Z \big]^c \Big)\cdot \prod_{ t \,\le \,\cdot} \big( 1 + \Delta Z (t) \big)  \exp \pare{ - \Delta Z (t)}  
\]
for the \textsl{stochastic exponential} of a scalar semimartingale $Z$ with $Z(0)=0\,$; note that it satisfies the   integral equation
$\, 
\Exp(Z)    \,=\, 1 + \int_0^{\, \cdot} \Exp(Z) (t-) \, \mathrm{d} Z(t)\,$.

\subsection{Strictly positive local martingale deflators} 

We     define $\,\Y  \,$ as the collection of all \emph{strictly positive} local martingales $Y$ with $Y(0) = 1$, such that $\,Y X^i\,$ is a local martingale for all $\iii$. The next result gives conditions on the drift and local covariance structure of $X$ which are equivalent to the requirement of non-emptiness for $\Y$.

\begin{thm} \label{thm:struct}
In the above set-up, the following two  conditions are equivalent:
\begin{enumerate}
	\item $\Y \neq \emptyset$.
	\item There exists a $d$-dimensional, predictable process $\rho$, such that $a = c \rho$ holds $\pare{\prob \otimes G}$-a.e. and the process $\int_0^\cdot \inner{\rho(t)}{c(t) \rho(t)} \ud G(t)$ is   finitely-valued.
\end{enumerate}
\end{thm}

The structural conditions in statement (2) of Theorem \ref{thm:struct} have appeared previously---see, for example, \cite{MR1353193} or \cite[Theorem 4.2, page 12]{MR1640352}.  
A proof of Theorem \ref{thm:struct} can be found in \cite[Section 4]{MR2732838}. We shall not repeat it here, but will provide some discussion in order to introduce   important quantities that will be used later on.

\subsection{Discussion of Theorem \ref{thm:struct}}

Let us start by assuming that condition (2) of Theorem \ref{thm:struct} holds. Since $a = c \rho$ implies that  $\inner{\rho}{a} = \inner{\rho}{c\rho} = \abs{\inner{\rho}{a}}$ holds $\pare{\prob \otimes G}$-a.e., it follows that $\rho$ is $X$-integrable, i.e., $\rho \in \Pre(X)$. Then,   the continuous-path semimartingale
\begin{equation} \label{eq:defl}
\hV \,\dfn  \,\Exp \pare{\int_0^\cdot \inner{\rho(t)}{\ud X(t)}} 
\end{equation}
is well-defined and satisfies the integral equation 
\begin{equation}
\label{eq:IE}
\hV \,=\, 1 + \int_0^\cdot \hV (t) \inner{\rho (t)}{\ud X(t)}.
\end{equation}
Straightforward computations  show now that $(1 / \hV)$ is a local martingale, as is $(X^i / \hV)$ for all $\iii$; consequently, $(1 / \hV) \in \Y$. In fact, whenever $L$ is a local martingale  with    $L(0) = 0$, $\Delta L > - 1$ and $[L, M] = 0\,$, the product $\big(1 / \hV \big)\, \Exp(L)$ is an element of $\Y\,$. Although we shall not make direct use of this fact, let us also note that every element  of $\Y\,$ is of  this product  form. 

\smallskip
 The argument of the preceding paragraph establishes the implication $(2) \Rightarrow (1)$ in Theorem \ref{thm:struct}. For completeness, we    discuss now briefly, how the failure of condition (2) implies the failure of condition (1) in Theorem \ref{thm:struct};   detailed arguments can be found in  \cite[Section 4]{MR2732838}. Two contingencies need to be considered:

\smallskip
\noindent
 {\it (i)   The vector     $a$   fails to be in the range of the matrix    $c\,$, on a predictable set $E$ of strictly positive $\pare{\prob \otimes G}$-measure.} In this case one can find $\,\zeta \in \Pre(X)\,$, such that $\,c \, \zeta = 0\,$ and the process $\int_0^\cdot \inner{\zeta (t)}{\ud X(t)} = \int_0^\cdot \inner{\zeta (t)}{a(t)} \ud G(t)$ is   nondecreasing everywhere and eventually strictly positive on $E$. This implies in a straightforward way that $\Y = \emptyset$. 
 
 \smallskip
\noindent
 {\it (ii)  \,A $d$-dimensional predictable process $\rho$ exists, so  that $a = c \rho$ holds $\pare{\prob \otimes G}$-a.e.; but the event $\big\{\int_0^T \inner{\rho(t)}{c(t) \rho(t)} \ud G(t) = \infty \big\}$ has positive probability for some $T > 0$.} Then, upon noting that $\rho \indic_{\set{|\rho | \leq n}} \in \Pre(X)$ holds  for all $\nin$, and defining $V_n = \Exp \pare{\int_0^\cdot \inner{\rho(t) \indic_{\set{|\rho(t)| \leq n}}}{\ud X(t)}}$, one can show that the collection $\set{V_n(T) \such \nin}$ is   unbounded in probability. This  
again implies   $\Y = \emptyset \,$. 

Indeed, if $\Y\,$ were not empty, then for any $Y \in \Y$   it would be straightforward to check that $Y V_n$ is a nonnegative local martingale---thus, a supermartingale---for all $\nin$. By Doob's maximal inequality, this would imply that $\set{Y(T) V_n(T) \such \nin}$ is bounded in probability;  and  since $Y$ is strictly positive, that $\set{V_n(T) \such \nin}$ is also bounded in probability. But we have already seen that the opposite is true.

\begin{rem}
\label{rem:less_check}
{\it A Reduction:} 
Under condition (2) of Theorem \ref{thm:struct},    a given     $d$-dimensional and predictable process $H$ belongs to $\Pre(X)$ if and only if $\, \int_0^\cdot \inner{H(t)}{c(t)H(t)} \ud G(t)\,$ is finitely-valued. Indeed, the Cauchy-Schwartz inequality (see also \cite[Proposition 3.2.14]{MR1121940}) and the $(\prob \otimes G)$-identity $\,a = c \rho\,$ imply then 
\[
\int_0^\cdot \abs{ \inner{H(t)}{a(t)} } \,\ud \, G(t) \leq \pare{\int_0^\cdot \inner{H(t)}{c(t) H(t)} \ud G(t)}^{1/2} \pare{\int_0^\cdot \inner{\rho(t)}{c(t) \rho(t)} \ud G(t)}^{1/2},
\]
and show that $\int_0^\cdot \abs{ \inner{H(t)}{a(t)} } \ud G(t)$ is  \emph{a fortiori} finitely-valued. 

In obvious notation,   we have $\Pre(X) = \Pre(M)$ under the condition (2) of Theorem \ref{thm:struct}.
\end{rem}

\begin{rem}
{\it An Interpretation:} 
 It follows from \eqref{eq:IE} that  
 the process  $\hV$ can be interpreted as the strictly positive wealth  generated by the predictable process  $\rho$ viewed as a ``portfolio'', starting with a unit of capital. 
The components of $X$ represent then the \emph{returns} of  the various assets in an equity  market;   
the strictly positive processes $S_i = \mathcal{E} (X_i) $ are the {\it prices} of these assets;  the components of $\rho$ stand for the {\it proportions} of current wealth  invested in each one of these assets; whereas the scalar processes $\, \vartheta_i = (\hV / S_i) \rho_i\,$ (respectively, $\, \eta_i =  \hV    \rho_i$) keep track of the numbers of {\it shares} (resp., amounts of currency) invested in the various assets. 
\end{rem}

\subsection{Optional Decomposition Theorem} 
\label{sec:odt}

The following is the main result of this work. It is   proved    in Section \ref{sec:proof}.

\begin{thm} 
\label{thm:odt}
Assume that $\Y \neq \emptyset$. Let $V$ be an adapted \cadlag \ process,   
locally bounded from below. Then, the following statements are equivalent:
\begin{enumerate}
	\item The product $Y V$ is a local supermartinagle, for all $\,Y \in \Y$.
	\item The process $V$ is of the form 
	\begin{equation}
\label{eq:dec}
V\, = \,V(0) + \int_0^{\,\cdot} \inner{H(t)}{\ud X(t)} -  C\,,
\end{equation}
	 where $H \in \Pre(X)$   
	 and $C$ is a nondecreasing,  right-continuous and adapted process with $C(0) = 0$, which  is locally bounded from above.
\end{enumerate}
\end{thm}

\begin{rem}
{\it On Uniqueness:} 
In the present setting, the decomposition \eqref{eq:dec} of $V$ is unique in the following sense: If $V\, = \,V(0) + \int_0^{\,\cdot} \inner{H'(t)}{\ud X(t)} -  C'$ holds along with \eqref{eq:dec} for $H' \in \Pre(X)$ and $C'$ a nondecreasing,  right-continuous and adapted process with $C(0) = 0$, then $C= C'$ and $\int_0^{\,\cdot}\inner{H(t)}{\ud X(t)} = \int_0^{\,\cdot}\inner{H'(t)}{\ud X(t)}$ hold modulo evanescence. 

Indeed,  let $\hY \dfn 1 / \hV$, $D = C' - C$, $F \dfn   H' -H$, and note that $\hY D = \hY \int_0^{\,\cdot}\inner{F(t)}{\ud X(t)}$ is a continuous-path local martingale. This follows from integration-by-parts on the right-hand side of the last equality, the equation \eqref{eq:IE}  for $\hV$, and property (2) in Theorem \ref{thm:struct} for $\rho$. (For the latter local martingale property, see also the proof of the implication $(2) \Rightarrow (1)$ in Theorem \ref{thm:odt} in the beginning of \S \ref{subsec:proof_odt}.)  
Integrating by parts again, we see that $\int_0^\cdot \hY(t) \, \ud D(t) = \hY D - \int_0^\cdot D(t-) \,\ud \hY(t)$ is both a continuous-path local martingale and a finite-variation process, which implies   $\int_0^\cdot \hY(t) \,\ud D(t) = 0$ modulo evanescence. Since $\hY$ is strictly positive, the last fact implies that $D = 0$ holds modulo evanescence,   completing the argument. 

Finally, let us note that for the integrands $H$ and $H'$  we may only conclude that $H = H'$ holds in the $\pare{\prob \otimes G}$-a.e. sense.   
\end{rem}

\section{Proof of Theorem \ref{thm:odt}} \label{sec:proof}

We  shall  denote by $\mathcal{L}^\mathrm{c}$ the collection of all local martingales $L$ with continuous paths and $L (0) = 0\,$; furthermore, we  shall  denote by $\Lc$    the collection of all  local martingales $L$ with  $\Delta L > - 1$ and $L (0) = 0\,$ which are purely discontinuous, i.e., satisfy $[L,\Lambda]\equiv0$ for all $\Lambda \in  \mathcal{L}^\mathrm{c}$. 

\subsection{An intermediate result}

In order to prove Theorem \ref{thm:odt}, we first isolate the result that will enable us   eventually to deal only with continuous-path local martingales. 

\begin{lem} \label{lem:key}
Let $B$ be a locally bounded from above semimartingale  with the following properties:
\begin{itemize}
	\item $B + [B, L]$ is a local submartingale, for every    $L \in \Lc$. 
	\item $[B, L] = 0$, for every   $L \in \mathcal{L}^\mathrm{c}$.
\end{itemize}
Then, $B$ is actually non-decreasing.
\end{lem}

\begin{proof} The first property implies  that $B$ itself is a local submartingale (just take $L \equiv 0$ there). Replacing $B$ by $B - B(0)$, we may assume that $B(0) = 0$. Furthermore,   standard localization arguments imply that we may take   $B$     to be bounded from above; therefore,  we shall assume in the   proof the existence of $\,b \in \Real_+$   such that $B \leq b$. This means that $B$ is an actual submartingale with   last element $B (\infty)$, and that $\expec \bra{B(T)} \geq \expec \bra{B(0)} = 0$ holds for all stopping times $T$.  

Consider now a countable collection $(\tau_n)_{\nin}$ of \emph{predictable} stopping times that exhaust the accessible jump-times of $B$. Defining also the predictable set
\[
\,J \,\dfn \,\bigcup_{\nin} \, \dbra{\tau_n, \tau_n}\,,
\]
we note that the process $$B^J \,\dfn \,\int_{(0, \cdot \,]} \indic_J(t)\, \ud B(t)\, = \,\sum_{\nin} \Delta B (\tau_n) \,\indic_{\{ \tau_n \, \leq \,\cdot \}}$$ is a local submartingale. 

\smallskip
\noindent
$\bullet~$  We shall first show that $B^J$ is nondecreasing, which amounts to showing that $\Delta B (\tau_n) \geq 0$ holds for all $\nin$. To this end, we define for each $\nin$    the $\F (\tau_n -)$-measurable random variable 
$$
p_n \, \dfn \, \prob \bra{\,\Delta B (\tau_n) < 0 \such \F (\tau_n -)\,} \indic_{\{ \tau_n < \infty \}} \,.
$$ 
On $\set{\tau_n < \infty, \, p_n = 0}$, we have $\Delta B (\tau_n) \geq 0\,$. 
For $\nin$ and $\kin$, we define $L_{n, k} \in \Lc$ to be the local martingale with $L_{n, k}(0) = 0$ and a single jump at $\,\tau_n\,$, such that   
$$
\Delta L_{n, k} (\tau_n) = \big(1 - (1/k) \big) \, \indic_{\{p_n > 0\}} \pare{(1 / p_n) \,\indic_{\{\Delta B(\tau_n) < 0\}} - 1}\,.
$$ 
We note that $B + [B, L_{n, k}] = B$ holds on $\dbraco{0, \tau_n}\,$, while on the event $\, \set{\tau_n < \infty\,, \, p_n > 0}\,$ we have  
\[
\Delta \pare{B + [B, L_{n, k}]} (\tau_n) \,=\, \Delta B (\tau_n) \big( 1 + \Delta L_{n, k} (\tau_n) \big) \,=\, \frac{\,\Delta B (\tau_n) \,}{k} - \frac{\,1 - (1/k)\,}{p_n} \pare{\Delta B (\tau_n)}^-.   
\]
The properties imposed on $B$ imply, in particular,  that $B + [B, L_{n,k}]$ is a local submartingale, bounded from above by $\,b + 1/k\,$ on $[0, \tau_n]$. It follows that on the event  $\set{\tau_n < \infty, \, p_n > 0}$  we have $\expec \bra{\,\Delta B (\tau_n) \pare{1 + \Delta L_{n,k} (\tau_n)} \such \F (\tau_n -)\,}  \geq 0$, which translates into
 $$
 \expec \bra{ \pare{\Delta B (\tau_n)}^- \such \F (\tau_n -)} \leq \,  \frac{\,\expec \bra{ \pare{\Delta B (\tau_n)}^+ \such \F (\tau_n -)}\,}{1 +  \big( (k-1) / p_n \big)} \qquad \text{on the event } \quad \set{\tau_n < \infty\,, \,\,p_n > 0}
 $$ 
 for all $\kin$. Sending $\,k \to \infty \,$, it follows that $\Delta B(\tau_n) \geq 0$ holds on $\set{\tau_n < \infty, \,p_n > 0}$.

\smallskip

\smallskip
\noindent
$\bullet~$  It remains to show that $B' \dfn B - B^J = \int_{(0, \cdot]} \indic_{\pare{\Omega \times \Real_+} \setminus J}(t) \ud B(t)$ is also a nondecreasing process. We note that $B'$ inherits some of the   properties of $B\,$: in particular, we have $\, B' (0)=0$ 
 and for every $L \in \Lc$, the process 
$$
B' + [B', L] \,=\,  
+ \int_0^\cdot \indic_{\pare{\Omega \times \Real_+} \setminus J}(t) \, \big( \ud  B(t) +  \ud[B, L](t)\big)
$$ 
is a local submartingale. We also have that $B' \leq B \leq b$, and additionally $B'$ has jumps only at totally inaccessible stopping times. To ease the notation we write $B$ instead of $B'$ for the remainder of this proof, assuming from now onwards that $B$ has jumps only at totally inaccessible stopping times. For each $\nin$, we define the local martingale 
$$\,
L_n \, \dfn \, n \sum_{\,t \, \leq \, \cdot} \, \indic_{ \{\Delta B (t) \leq - 1/n\}} - D_n\, ,
$$
 where $D_n$ is a suitable nondecreasing process with continuous paths (since the jumps of $B$ are totally inaccessible). Note that $L_n \in \Lc$, which implies that $B + [B, L_n]$ is a local submartingale for all $\nin$. Furthermore, we note that  
 $B + [B, L_n] = B + n \sum_{\,t\, \leq \,\cdot} \Delta B (t) \, \indic_{\{\Delta B (t) \leq - 1/n\}} \leq b\,$, which implies that $B + [B, L_n]$ is a true submartingale. Therefore, it follows that
\[
n  \, \expec \bra{- \sum_{t \leq n} \Delta B (t) \,\indic_{\{\Delta B (t)\leq - 1/n \}}  } \, \leq \, \expec \bra{B(n) - B(0)} \, =\, \expec \bra{B(n)  } \, \leq \, b\,, \quad \forall \,\,\nin.
\]
The monotone convergence theorem  gives now $\,\expec\big[  \sum_{\tir} \pare{\Delta B (t)}^- \big] = 0$, which implies   $\Delta B \geq 0$. Continuing, we  define for each $n \in \mathbb{N}\,$ a new local martingale 
$$
\widetilde{L}_n \dfn \widetilde{D}_n - \big(1 - (1/n) \big)\, \sum_{t \, \leq \, \cdot} \, \indic_{ \{\Delta B (t) \geq 1/n \}}\,,
$$ 
where $\widetilde{D}_n$ is an   appropriate continuous and nondecreasing process. Note that we have $\,\widetilde{L}_n \in \Lc\,$ for all $\,\nin\,$, which implies that the processes $$\,B + [B, \widetilde{L}_n] \,= \,B - \big(1 - 1/n\big) \,\sum_{\,t \,\leq \,\cdot} \Delta B(t) \, \indic_{\{\Delta B(t) \geq 1/n\}}\,, \qquad  \nin$$    are local submartingales, 
uniformly bounded from above by $b\,$. Thus, it  follows that   $\,\sum_{\,t \,\leq \,\cdot} \, \Delta B(t)\,$ is finitely-valued; and that $\widetilde{B} := B - \sum_{\,t \, \leq \,\cdot} \Delta B (t)$ is a local submartingale. Recalling that the jumps of $B$ occur only at totally inaccessible stopping times, we see that this last process $\widetilde{B}$ has continuous paths and is strongly orthogonal to all continuous-path local martingales, which means that it is of finite variation. Since it is a local submartingale, it has to be nondecreasing. It follows from this reasoning that the process $\,B = \pare{B - \sum_{\,t  \,\leq \,\cdot} \Delta B (t)} + \sum_{\,t \,\leq \, \cdot} \Delta B (t)\,$ is nondecreasing, and this concludes the argument.
\end{proof}

\subsection{Proof of Theorem \ref{thm:odt}} \label{subsec:proof_odt}

The implication $(2) \Rightarrow (1)$ is straightforward. Indeed, we fix $Y \in \Y$  and note that $$Z \,\dfn  \, \int_0^\cdot Y(t-) \,\ud X(t) + [Y, X] \,=\, Y X - X(0) 
 - \int_0^\cdot X(t) \, \ud Y(t)$$ is a $d$-dimensional local martingale. Then, if $V = V(0) + \int_{0}^\cdot \inner{H(t)}{ \ud X(t)} - C$, one computes   
\[
Y V = V(0) + \int_0^\cdot V(t-)\, \ud Y(t) + \int_0^\cdot \inner{H(t)}{\ud Z(t)} - \int_0^\cdot Y(t) \,\ud C(t),  ~~~ 
\]
which shows that $Y V$ is a local supermartingale.   

\smallskip
\noindent 
$\bullet~$ 
For the implication $(1) \Rightarrow (2)$, let us assume that $V$ is such that $Y V$ is a local supermartingale for all $Y \in \Y$. In particular, recalling the notation of \eqref{eq:defl}, we note that $(V / \hV)$ is a local supermartingale  and write this process in its Kunita-Watanabe\,/\,Doob-Meyer representation
\begin{equation}  
\label{eq:Kuni-Wata}
U \, \dfn \,   V / \hV  \, = \,V(0) + \int_0^\cdot \inner{\theta (t)}{ \ud M (t)} + N - B.
\end{equation}
Here  $\theta \in \mathcal{P} (X)$ (see Remark \ref{rem:less_check}), and $N \in \mathcal{L}^\mathrm{c}\,$ satisfies  $[N, M ] = 0\,$, whereas $B$ is a local submartingale with $B(0) = 0$  and ``purely discontinuous'',  in the sense that 
\begin{equation}  
\label{eq:*}
[B, L ] = 0 \quad \text{holds for every}\,\,\,   L \in \mathcal{L}^\mathrm{c}\,.
\end{equation}
 In particular, we note that $[B, N ] = 0$. 

 \smallskip
 \noindent {\it (i):} 
The first item of business is to show that the process $B\,$ in (\ref{eq:Kuni-Wata}) is actually {\it nondecreasing;}   for this we shall use Lemma \ref{lem:key}. Since $N + \int_0^\cdot \inner{\theta (t)}{ \ud M (t)}$ is continuous   and $U$   locally bounded from below, the process  $B = V(0) + N + \int_0^\cdot \inner{\theta (t)}{ \ud M (t)} - U$ is locally bounded from above.

Let us fix $L \in \Lc$. From $\,(1 / \hV\,) \,\Exp(L) \in \Y\,$    
we observe -- e.g., using the product rule -- that the process $$
V \cdot \big(1 / \hV\big) \,\Exp(L) \,=\, U \, \Exp(L)\,=\,\pare{V(0) + N + \int_0^\cdot \inner{\theta (t)}{ \ud M (t)} - B} \Exp(L)
$$ 
is a local supermartingale. Furthermore,   the process $\pare{N + \int_0^\cdot \inner{\theta (t)}{ \ud M (t)}} \Exp(L)$ is a local martingale, and it follows from these two observations that 
\[
B \,\Exp(L)\, = \int_0^\cdot \Exp(L) (t-) \,\ud \pare{B + [B, L]} (t)+ \int_0^\cdot \pare{B \Exp(L)} (t-) \, \ud L(t) 
\]
is a local submartingale. This, in turn, implies that 
\[
B + [B, L] \,= \int_0^\cdot \frac{1}{\Exp(L) (t-)} \, \ud \pare{B L} (t) - \int_0^\cdot B(t-) \,\ud L(t)
\]

\medskip
\noindent
is also a local submartingale. Recalling the property (\ref{eq:*}) and invoking Lemma \ref{lem:key}, we conclude from this observation  that the local submartingale  $B$ in the decomposition (\ref{eq:Kuni-Wata})  is indeed non-decreasing.

\smallskip
 \noindent {\it (ii):} 
 The second  item of business is to show that $N \equiv 0\,$ holds in \eqref{eq:Kuni-Wata}.  The crucial observation here  is that, because of $\,[N, M] = 0\,$, the product  $\,(1 / \hV) \, \Exp(n N)\,$ is an element of $\,\Y\,$ for all $\,n \in \mathbb{N}$. As a consequence,   $\,U \Exp(n N) = V \cdot (1 / \hV) \, \Exp(n N)\,$ is a local supermartingale for all $\nin$. 
 
 \smallskip
 Since $[\Exp(n N), B] = 0$ and $[\Exp(n N), M] = 0$, it follows that $\Exp(n N) (B - N)$ is a local submartingale for all $\nin$. Now we observe
\[
\Exp(n N) (B -N) = \int_0^\cdot (B - N) (t-) \,\ud \Exp(n N) (t) + \int_0^\cdot \Exp(n N) (t-) \,\ud \big(B - N - [n N, N]\big) (t), 
\]
from which it follows that $B - n [N, N]$ is a local submartingale for all $\nin$. This is only possible if $[N, N] = 0$ which, since $N(0) = 0$, implies $N \equiv 0$; as a result, \eqref{eq:Kuni-Wata} becomes 
\begin{equation} \label{eq:Kuni-Wata-2}
U \,= \,(V / \hV) \,= \,V(0) + \int_0^\cdot \inner{\theta (t)}{\ud M (t)} - B .
\end{equation}

 \noindent {\it (iii):} 
 We are now in a position to conclude. Yet another application of the integration-by-parts formula on $V  = \hV \, U$ gives, in conjunction with \eqref{eq:IE}, \eqref{eq:Kuni-Wata-2} and Theorem \ref{thm:struct}, 
 the decomposition 
\begin{align*}
V &= V(0) + \int_0^\cdot U(t-)\, \ud \hV(t) + \int_0^\cdot \hV(t )\, \ud U (t) + [\hV, U] \\
&= V(0) + \int_0^\cdot U(t-)\, \ud \hV(t) + \int_0^\cdot \hV(t ) \inner{\theta(t)}{\,\ud X (t)} - \int_0^\cdot \hV(t) \, \ud B (t)
\\
&= V(0) + \int_0^\cdot \hV(t) \inner{U(t-) \rho(t) + \theta(t)}{\ud X (t)} - \int_0^\cdot \hV(t) \, \ud B (t).
\end{align*}
Defining $\, U_{-} (t): = U(t-) \,$ for $\, t >0\,$, as well as   
\[
H \,\dfn \,\hV \big( U_- \,\rho + \theta \big)  \, \in \, \Pre(X)\, \quad \text{ and }\quad \,C\, \dfn \int_0^\cdot \hV(t) \,\ud B (t)\,,
\]
we obtain the   decomposition \eqref{eq:dec} as claimed.  \qed

\bibliography{odt}
\bibliographystyle{alpha}
\end{document}